%% file: root.tex
\documentclass[letterpaper, 10 pt, conference]{ieeeconf}
\IEEEoverridecommandlockouts                              
\makeatletter
\let\NAT@parse\undefined
\makeatother

\usepackage{graphicx}
\usepackage{amsmath,amssymb}
\usepackage{mathtools}
\usepackage{cite}
\usepackage{color}
\usepackage{booktabs,multirow}
\usepackage[flushleft]{threeparttable}
\usepackage{caption}
\usepackage{fancyhdr}
\usepackage{xcolor}

\usepackage{hyperref}
\hypersetup{
	colorlinks,
	linkcolor={blue!60!black},
	citecolor={blue!60!black},
	urlcolor={blue!60!black}
}

\def\doi{10.23919/ecc65951.2025.11186977}
\fancyfootoffset{-1em}
\fancypagestyle{copyright}{
	\cfoot{\vspace*{0em}\footnotesize This paper was presented at the 2025 23rd European Control Conference (ECC). \href{http://dx.doi.org/\doi}{DOI \doi}. \textcopyright{}2025 EUCA.}
}

\newtheorem{ass}{Assumption}

\newtheorem{thm}{Theorem}
\newtheorem{prop}{Proposition}
\newtheorem{lem}{Lemma}
\newtheorem{cor}{Corollary}
\newtheorem{rem}{Remark}

\title{\LARGE \bf
	Optimal state estimation: Turnpike analysis and performance results
}

\author{Julian D. Schiller, Lars Grüne, and Matthias A. Müller
\thanks{This work was supported by the Deutsche Forschungsgemeinschaft (DFG, German Research Foundation), Projects 426459964 and 499435839.} 
\thanks{Julian D. Schiller and Matthias A. Müller are with the Leibniz University Hannover, Institute of Automatic Control, 30167 Hannover, Germany. Lars Grüne is with the Chair of Applied Mathematics, University of Bayreuth, 95447 Bayreuth, Germany.
	{\tt\small schiller@irt.uni-hannover.de},
	{\tt\small lars.gruene@uni-bayreuth.de},
	{\tt\small mueller@irt.uni-hannover.de}.}%
}

\begin{document}

\maketitle
\thispagestyle{empty}
\thispagestyle{copyright}

\begin{abstract}
	In this paper, we introduce turnpike arguments in the context of optimal state estimation. In particular, we show that the optimal solution of the state estimation problem involving all available past data serves as turnpike for the solutions of truncated problems involving only a subset of the data.	We mathematically formalize this phenomenon and derive a sufficient condition that relies on a decaying sensitivity property of the underlying nonlinear program. As second contribution, we show how a specific turnpike property can be used to establish performance guarantees when approximating the optimal solution of the full problem by a sequence of truncated problems, and we show that the resulting performance (both averaged and non-averaged) is approximately optimal with error terms that can be made arbitrarily small by an appropriate choice of the horizon length. In addition, we discuss interesting implications of these results for the practically relevant case of moving horizon estimation and illustrate our results with a numerical example.
\end{abstract}

\section{Introduction}
Reconstructing the internal state trajectory of a dynamical system based on a batch of measured input-output data is an important problem of high practical relevance.
This can be accomplished, for example, by solving an optimization problem to find the best state and disturbance trajectories that minimize a suitably defined cost function depending on the measurement data.
If all available data is taken into account, this corresponds to the \emph{full information estimation} (FIE) problem.
However, if the data set or the underlying model is very large or only a limited amount of computation time or resources is available (as is the case with, e.g., online state estimation), the optimal solution to the FIE problem is usually difficult (or even impossible) to compute in practice.
For this reason, it is essential to find a reasonable approximation, e.g., using a sequence of truncated optimal estimation problems, each of which uses only a limited time window of the full data set.
In the case of online state estimation, this corresponds to \emph{moving horizon estimation} (MHE), where at each discrete time step an optimal state estimation problem is solved with a data set of fixed size.

Current research in the field of MHE is primarily concerned with stability and robustness guarantees, see, e.g., \cite[Ch.~4]{Rawlings2017} and \cite{Allan2021a,Knuefer2023,Schiller2023c,Hu2023}. These works essentially show that under suitable detectability conditions, the estimation error of MHE converges to a neighborhood of the origin, whose size depends on the true unknown disturbance. However, results on the actual performance of MHE, and in particular on the approximation accuracy and performance loss compared to the desired optimal (but unknown) FIE solution, are lacking.

Whereas performance guarantees for state estimators are generally rather rare and usually restricted to linear systems (cf., e.g.,~\cite{Sabag2022,Brouillon2023}), they often play an important role in nonlinear optimal control, especially when the overall goal is an economic one.
Corresponding results usually employ a \emph{turnpike} property of the underlying optimal control problem, cf.~\cite{McKenzie1986,Carlson1991}.
This property essentially implies that optimal trajectories stay close to an optimal equilibrium (or in general an optimal time-varying reference) most of the time, which is regarded as the \emph{turnpike}.
Turnpike-related arguments are an important tool for assessing the closed-loop performance of model predictive controllers with general economic costs on finite and infinite horizons, cf. \cite{Gruene2016b,Faulwasser2022,Gruene2019}.
Necessary and sufficient conditions for the presence of the turnpike phenomenon in optimal control problems are discussed in, e.g.,~\cite{Gruene2016a,Trelat2023} and are usually based on dissipativity, controllability, and suitable optimality conditions.

Our contribution is twofold.
First, in Section~\ref{sec:turnpike} we introduce turnpike arguments in the context of optimal state estimation, which is a novelty in itself. In particular, we show that the optimal FIE solution involving all past data serves as turnpike for the solutions of truncated state estimation problems. We mathematically formalize this phenomenon and derive a sufficient condition relying on a decaying sensitivity property of the underlying nonlinear program.
Second, in Section~\ref{sec:performance_sec}, we provide novel performance guarantees for optimal state estimation.
To this end, we construct a candidate for approximating the FIE solution by a sequence of truncated problems, which can be solved in parallel for offline estimation. Assuming that the optimal state estimation problem exhibits a certain turnpike property, we show that the performance of this candidate is approximately optimal, both averaged and non-averaged, with error terms that can be made arbitrarily small by an appropriate choice of horizon length.
In addition, we discuss interesting implications of these results for MHE.

\subsubsection*{Notation}
We denote the set of integers by $\mathbb{I}$, the set of all integers greater than or equal to $a \in \mathbb{I}$ by $\mathbb{I}_{\geq a}$, and the set of integers in the interval $[a,b]\subset\mathbb{I}$ by $\mathbb{I}_{[a,b]}$.
The Euclidean norm of a vector $x\in\mathbb{R}^n$ is denoted by $|x|$, and the weighted norm with respect to a positive definite matrix $Q=Q^\top$ by $|x|_Q=\sqrt{x^\top Q x}$. The maximum eigenvalue of $Q$ is denoted by $\lambda_{\max}(Q)$.
We refer to a sequence $\{x_j\}_{j=a}^b$, $x_j\in\mathbb{R}^n$, $j\in\mathbb{I}_{[a,b]}$ with $x_{a:b}$.
Finally, we recall that a function $\alpha:\mathbb{R}_{\geq 0}\rightarrow\mathbb{R}_{\geq 0}$ is of class $\mathcal{K}$ if it is continuous, strictly increasing, and satisfies $\alpha(0)=0$.
By $\mathcal{L}$, we refer to the class of functions $\theta:\mathbb{R}_{\geq 0}\rightarrow \mathbb{R}_{\geq 0}$ that are continuous, non-increasing, and satisfy $\lim_{s\rightarrow\infty}\theta(s)=0$, and by $\mathcal{KL}$ to the class of functions $\beta:\mathbb{R}_{\geq 0}\times\mathbb{R}_{\geq 0}\rightarrow\mathbb{R}_{\geq 0}$ with $\beta(\cdot,s)\in\mathcal{K}$ and $\beta(r,\cdot)\in\mathcal{L}$ for any fixed $s\in\mathbb{R}_{\geq 0}$ and $r\in\mathbb{R}_{\geq 0}$, respectively.

\section{Problem Setup}\label{sec:problem}

\subsection{System description}
We consider the following system
\begin{subequations}\label{eq:sys}
	\begin{align}
		{x}_{t+1} &= f(x_t,u_t,w_t),\label{eq:sys_1}\\
		y_t &= h(x_t,u_t) + v_t\label{eq:sys_2}
	\end{align}
\end{subequations}
with discrete time $t\in\mathbb{I}_{\geq0}$, state $x_t\in\mathbb{R}^n$, (known) control input $u_t\in\mathbb{R}^m$, (unknown) process disturbance $w_t\in\mathbb{R}^q$, (unknown) measurement noise $v_t\in\mathbb{R}^p$, and noisy output measurement $y_t\in\mathbb{R}^p$.
The functions $f:\mathbb{R}^n\times\mathbb{R}^m\times\mathbb{R}^q\rightarrow \mathbb{R}^n$ and $h:\mathbb{R}^n\times\mathbb{R}^m\rightarrow \mathbb{R}^p$ define the system dynamics and output equation, which we assume to be continuous.
We further assume that trajectories of the system~\eqref{eq:sys} satisfy
\begin{equation*}
		(x_t,u_t,w_t,v_t,y_t) \in \mathcal{X}\times\mathcal{U}\times\mathcal{W}\times\mathcal{V}\,\times\,\mathcal{Y},\  t\in\mathbb{I}_{\geq0}
\end{equation*}
for some known sets $\mathcal{X}\subseteq\mathbb{R}^n$, $\mathcal{U}\subseteq\mathbb{R}^m$, $\mathcal{W}\subseteq\mathbb{R}^q$, $\mathcal{V}\subseteq\mathbb{R}^p$, $\mathcal{Y}\subseteq\mathbb{R}^p$.
Such knowledge typically arises from the physical nature of the system (e.g., non-negativity of certain physical quantities such as partial pressure or absolute temperature), the incorporation of which can significantly improve the estimation results, cf., e.g.,~\cite[Sec.~4.4]{Rawlings2017} and \cite{Schiller2023c}.

\subsection{Optimal state estimation}
For ease of notation, we define the input-output data tuple $d_t := (u_t,y_t)$, $t\in\mathbb{I}_{\geq0}$. Now, consider a given data batch~$d_{0:T}$ for some $T\in\mathbb{I}_{\geq0}$.
We aim to compute the state and disturbance sequences $\hat{x}_{0:T}$ and $\hat{w}_{0:T-1}$ that are optimal in the sense that they minimize a cost function involving the full data set $d_{0:T}$.
In particular, we consider the following FIE optimization problem $P_T(d_{0:T})$:
\begin{subequations}\label{eq:MHE}
	\begin{align}\label{eq:MHE_0}
		&\min_{\hat{x}_{0:T},\, \hat{w}_{0:T-1}}\  J_{T}(\hat{x}_{0:T},\hat{w}_{0:T-1}; d_{0:T}) \\ 
		\text{s.t. }\
		&\hat{x}_{j+1}=f(\hat{x}_{j},u_j,\hat{w}_{j}), \ j\in\mathbb{I}_{[0,T-1]}, \label{eq:MHE_f} \\
		&\hat{x}_{j}\in\mathcal{X}, \ y_j-h(\hat{x}_{j},u_j)\in\mathcal{V}, \ j\in\mathbb{I}_{[0,T]}, \label{eq:MHE_x} \\
		&\hat{w}_{j}\in\mathcal{W},\ j\in\mathbb{I}_{[0,T-1]}.\label{eq:MHE_w}
	\end{align}
\end{subequations}
For convenience, we define the combined sequence $\hat{z}_{0:T}$ satisfying $\hat{z}_j=(\hat{x}_j,\hat{w}_j)$ for $j\in\mathbb{I}_{[0,T-1]}$ and $\hat{z}_T = (\hat{x}_T,0)$. The constraints~\eqref{eq:MHE_f}--\eqref{eq:MHE_w} enforce the prior knowledge about the system model, the domain of the true trajectories, and the disturbances/noise (feasibility is always guaranteed due to our standing assumptions).
We consider the cost function
\begin{equation}\label{eq:MHE_cost}
	J_T(\hat{x}_{0:T},\hat{w}_{0:T-1} ;d_{0:T}) {\,=}\hspace{-0.1ex} \sum_{j=0}^{T-1}\hspace{-0.3ex}l(\hat{x}_{j},\hat{w}_{j};d_j){\,+\,} g(\hat{x}_{T};d_T)
\end{equation}
with continuous stage cost $l:\mathcal{X}\times\mathcal{W}\times\mathcal{U}\times\mathcal{Y}\rightarrow\mathbb{R}_{\geq0}$ and terminal cost $g:\mathcal{X}\times\mathcal{U}\times\mathcal{Y}\rightarrow\mathbb{R}_{\geq0}$.
In the state estimation context, a cost function with terminal cost is usually referred to as the \textit{filtering form} of the state estimation problem, see also~\cite[Ch.~4]{Rawlings2017} for a discussion on this topic.
Note that~\eqref{eq:MHE_cost} is a generalization of classical designs for state estimation, where $l$ and $g$ are positive definite in the disturbance input~$\hat{w}$ and the fitting error $y-h(\hat{x},u)$, cf.~\cite[Ch.~4]{Rawlings2017}; it particularly includes the practically relevant case of quadratic cost functions
	\begin{align}
	l(x,w;d) &= |w|^2_Q + |y-h(x,u)|^2_R, \label{eq:stage_cost}\\
	g(x;d) &= |y-h(x,u)|^2_G, \label{eq:term_cost}
\end{align}
where $Q,R,G$ are positive definite weighting matrices.
However, all results obtained in this paper also hold for more general cost functions $l$ and $g$, which allow the objective~\eqref{eq:MHE_cost} to be tailored to the specific problem at hand.

The FIE problem $P_T$ in~\eqref{eq:MHE} is a parametric nonlinear program, the solution of which solely depends on the (input-output) data provided, i.e., the sequence $d_{0:T}$.
We characterize solutions to $P_T$ using a generic mapping $\zeta_T:\mathbb{I}_{[0,T]}\times(\mathcal{U}\times\mathcal{Y})^{T+1}\rightarrow\mathcal{X}\times\mathcal{W}$ that maps data to optimal states and disturbance inputs as
\begin{equation}
	z^*_j = \zeta_T(j,d_{0:T}), \ j\in\mathbb{I}_{[0,T]}, \label{eq:z_T}
\end{equation}
with the value function $V_T(d_{0:T}) = J_T(x^*_{0:T},w^*_{0:T-1};d_{0:T})$.

If the data set or the underlying model is large and/or the computations are limited in terms of time or resources (as is the case in many practical applications), solving the FIE problem $P_T$ for the optimal solution is usually difficult (or even impossible).
For this reason, we also consider the truncated optimal state estimation problem  ${P_N(d_{\tau:\tau+N})}$ using a fixed horizon length $N\in\mathbb{I}_{[0,T]}$, i.e., the problem~$P_T$ in~\eqref{eq:MHE} with $T$ replaced by $N$ and the truncated data sequence $d_{\tau:\tau+N}$ for some $\tau\in\mathbb{I}_{[0,T-N]}$. We characterize corresponding solutions using the generic solution mapping $\zeta_N$:
\begin{equation}
	\hat{z}_{\tau+j}^* := \zeta_N(j,d_{\tau:\tau+N}),\ j\in\mathbb{I}_{[0,N]}. \label{eq:z_N}
\end{equation}
Moreover, with $\zeta_N^x$ we refer to the state variable defined by $\zeta_N$ such that $\hat{x}^*_{\tau+j} = \zeta_N^x(j,d_{\tau:\tau+N})$ for all $j\in\mathbb{I}_{[0,N]}$.

In the following, we investigate how the solution $\hat{z}^*_{\tau:\tau+N}$ of the truncated estimation problem $P_N(d_{\tau:\tau+N})$ behaves compared to the optimal FIE solution $z^*_{0:T}$ on the interval $\mathbb{I}_{[\tau,\tau+N]}$ (Section~\ref{sec:turnpike}), how it can be suitably approximated on $\mathbb{I}_{[0,T]}$ using a sequence of truncated problems $P_N$, and how good this approximation actually is (Section~\ref{sec:performance_sec}).

\section{Turnpike in optimal state estimation}\label{sec:turnpike}

Optimal state estimation problems (such as~$P_T$ in~\eqref{eq:MHE}) can generally be interpreted as optimal control problems using $\hat{w}$ as the control input, cf.~\cite[Sec.~4.2.3]{Rawlings2017} and \cite[Sec.~4]{Allan2020a}.
In particular, a cost function~\eqref{eq:MHE_cost} which is positive definite in the estimated disturbance and the fitting error (as, e.g., in \eqref{eq:stage_cost} and \eqref{eq:term_cost}) can be regarded as an economic output tracking cost associated with the ideal reference $(w^\mathrm{r}_j,y^\mathrm{r}_j) = (0,y_j)$, $j\in\mathbb{I}_{[0,T]}$.
This reference, however, is unreachable, as it is generally impossible to attain zero cost $V_T(d_{0:T})=0$ (except for the special case where $y_{0:T}$ corresponds to an output sequence of~\eqref{eq:sys} under zero disturbances $w_{0:T-1}\equiv0$, $v_{0:T}\equiv0$).
For unreachable references, on the other hand, it is known that the corresponding optimal control problem exhibits the turnpike property with respect to the \emph{best reachable reference}~\cite{Koehler2019}, which suggests that a similar phenomenon can also be expected in optimal state estimation.
In Section~\ref{sec:mot_example}, we provide a simple example that supports this intuition. Then, we formalize this behavior in Section~\ref{sec:turnpike_sensitivity} and draw connections to the concept of decaying sensitivity.

\begin{figure*}[t]
	\vspace{1.2ex}
	\hspace{1.2ex}
	\includegraphics{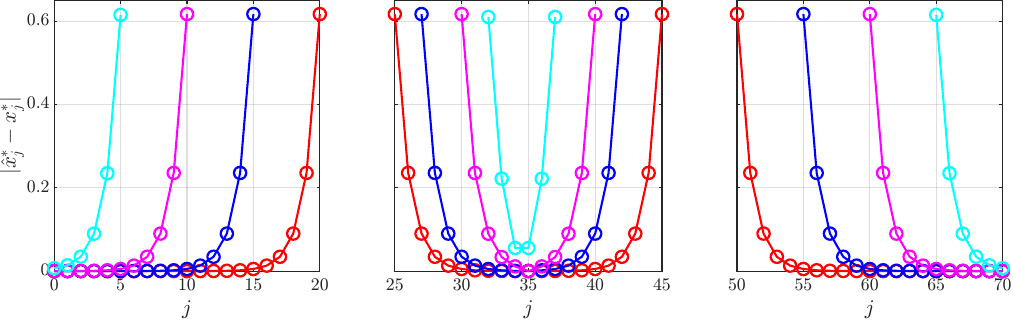}
	\caption{Difference between the optimal solution $x^*_j$ involving the full data batch with $T=70$ and the solution $\hat{x}^*_j$ of the truncated problem, plotted over $j\in\mathbb{I}_{[\tau,\tau+N]}$ for $N{=\,}5$ (cyan), $N{=\,}10$ (magenta), $N{=\,}15$ (blue), and $N{=\,}20$ (red) for $\tau{\;=\;}0$ (left), $\tau{\;=\;}\lfloor ({T{-\,}N})/2 \rfloor$ (middle), and $\tau{\;=\;}T{-\,}N$ (right).}
	\label{fig:OE}
\end{figure*}
	
\subsection{Motivating example}\label{sec:mot_example}
Suppose that the data $y_{0:T}$ is measured from the system $x_{t+1} = x_t + w_t$, $x_0=1$, $y_t = x_t + v_t$, where $w_t = v_t = 1$ for $t\in\mathbb{I}_{[0,T]}$ and $w_t = v_t = 0$ for $t\in\mathbb{I}_{\geq T+1}$.
We compare the solution of the optimal estimation problem $P_T(d_{0:T})$ employing the quadratic stage and terminal costs~\eqref{eq:stage_cost} and~\eqref{eq:term_cost} using $Q=R=G=1$ with the solution of the truncated problem $P_N(d_{\tau:\tau+N})$ for different values of $N$ and~$\tau$.

Figure~\ref{fig:OE} shows the difference between $x^*_j$ and $\hat{x}^*_j$ over $j\in\mathbb{I}_{[\tau,\tau+N]}$. 
Here, we find that the sequence $\hat{x}^*_{\tau:\tau+N}$ generally consists of three pieces: first, a transient at the beginning where $\hat{x}^*_j$ converges to ${x}^*_j$ (also called the \emph{approaching arc}); second, a large middle phase where $\hat{x}^*_j$ stays near ${x}^*_j$; third, a transient at the end where $\hat{x}^*_j$ diverges from ${x}^*_j$ (also called the \emph{leaving arc}).
Figure~\ref{fig:OE} indicates that these transients are independent of $N$.
For the special cases at the boundaries ($\tau=0$ and $\tau=T-N$), we observe one-sided transients; here, we refer to Remark~\ref{rem:motivation_TP} below for a detailed discussion.

Overall, it can be concluded that, when $N$ is large, the solution $\hat{x}^*_{j}$ remains close to the optimal solution ${x}^*_{j}$ for most of the time. This is in fact the key characteristic of any turnpike definition appearing in the \mbox{literature,~e.g.,~\cite{Gruene2016b,Faulwasser2022,Gruene2019,Gruene2016a,Trelat2023}.}

\subsection{Turnpike under a decaying sensitivity condition}\label{sec:turnpike_sensitivity}
Decaying sensitivity is a quite natural property of a parametric nonlinear program that characterizes how much perturbations in the data at one stage influence the corresponding solution at another stage, cf.~\cite{Na2020,Shin2022}.
In this section, we use this concept to derive a sufficient condition for the occurrence of turnpike behavior as observed in the preceding motivating example.

To this end, consider the auxiliary nonlinear program $\bar{P}(\bar{d}_{0:N},x^\mathrm{i},x^\mathrm{t})$ parameterized by data $(\bar{d}_{0:N},x^\mathrm{i},x^\mathrm{t})$, where $\bar{d}_j=(\bar{u}_j,\bar{y}_j)\in\mathcal{U}\times\mathcal{Y}$, $j\in\mathbb{I}_{[0,N]}$:
\begin{subequations}\label{eq:NLP}
	\begin{align}
		&\min_{\bar{x}_{0:N},\, \bar{w}_{0:N-1}} \ \sum_{j=0}^{N-1} l(\bar{x}_j,\bar{w}_j;\bar{d}_j)\\
	\text{s.t.} \ \ 
		&\bar{x}_0 = x^\mathrm{i}, \label{eq:MHE_con_init}\\
		&\bar{x}_{N} = x^\mathrm{t},\label{eq:MHE_con_term}\\
		&\bar{x}_{j+1} = f(\bar{x}_j,\bar{u}_j,\bar{w}_j),\  j\in\mathbb{I}_{[0,N-1]},\\
		&\bar{x}_{j}\in\mathcal{X}, \ \bar{y}_{j}-h(\bar{x}_{j},\bar{u}_j)\in\mathcal{V},\ j\in\mathbb{I}_{[0,N]},\\
		&\bar{w}_{j}\in\mathcal{W},\ j\in\mathbb{I}_{[0,N-1]}.\label{eq:MHE_con_2}
	\end{align}
\end{subequations}
Let $\bar{z}_j := (\bar{x}_j,\bar{w}_j)$, $j\in\mathbb{I}_{[0,N-1]}$, $\bar{z}_N := (\bar{x}_N,0)$.
We denote solutions to \eqref{eq:NLP} by $\bar{z}^*_j = \bar{\zeta}_N(j,\bar{d}_{0:N},x^\mathrm{i},x^\mathrm{t})$, $j\in\mathbb{I}_{[0,N]}$.

\begin{lem}\label{lem:optimality}
	Any minimizer of ${P}(d_{0:T})$ is a minimizer of $\bar{P}(d_{\tau:\tau+N},x^*_\tau,x^*_{\tau+N})$ for any $N\in\mathbb{I}_{[0,T]}$, $\tau\in\mathbb{I}_{[0,T-N]}$.
\end{lem}
\begin{proof}
	The proof relies on the principle of optimality.
	First, note that $\bar{P}(d_{\tau:\tau+N},x^*_\tau,x^*_{\tau+N})$ is feasible because the candidate solution $z^*_{\tau:\tau+N}$ satisfies the constraints \eqref{eq:MHE_con_init}--\eqref{eq:MHE_con_2} with $\bar{d}_{0:N}=d_{\tau:\tau+N}$.	
	Let $\bar{z}^*_{\tau+j} := \bar{\zeta}_N(j,d_{\tau:\tau+N},x^*_\tau,x^*_{\tau+N})$. Suppose the claim is false. Then, it must hold that
	\begin{equation*}
		\sum_{j=\tau}^{\tau+N-1}l(z_j^*;d_j)>\sum_{j=\tau}^{\tau+N-1}l(\bar{z}^*_j;d_{j}).
	\end{equation*}
	To both sides of this inequality, we add $\sum_{j=0}^{\tau-1}l(z_j^*;d_j) + \sum_{j=\tau+N}^{T-1}l(z_j^*;d_j) + g(x^*_T;d_T),$ which yields
	\begin{align}
		V_T(d_{0:T})>& \sum_{j=0}^{\tau-1}l(z_j^*;d_j) + \hspace{-0.5ex} \sum_{j=\tau}^{\tau+N-1}\hspace{-1ex}l(\bar{z}^*_j;d_{j}) + \hspace{-0.5ex} \sum_{j=\tau+N}^{T-1}\hspace{-0.7ex}l(z_j^*;d_j) \nonumber\\
		& + g(x^*_T;d_T). \label{eq:proof_VT}
	\end{align}
	Now, consider the sequence $\xi_{0:T}$ defined as
	\begin{equation}
		\xi_j = (\chi_j,\omega_j) := 
		\begin{cases}
			{z}^*_j, & j\in\mathbb{I}_{[0,\tau-1]}\\
			\bar{z}^*_j, & j\in\mathbb{I}_{[\tau,\tau+N-1]}\\
			{z}^*_j, & j\in\mathbb{I}_{[\tau+N,T]}. \label{eq:proof_xi}
		\end{cases}
	\end{equation}
	Then, \eqref{eq:proof_VT} can be re-written as
	\begin{equation}
		V_T(d_{0:T}) > \sum_{j=0}^{T-1}l(\xi_j;d_j) + g(\chi_T;d_T). \label{eq:proof_VT2}
	\end{equation}
	Note that $\xi_{0:T}$ as defined by~\eqref{eq:proof_xi} satisfies~\eqref{eq:MHE_f}--\eqref{eq:MHE_w} for all $j\in\mathbb{I}_{[0,T]}$ due to~\eqref{eq:MHE_con_init} and \eqref{eq:MHE_con_term}. However, \eqref{eq:proof_VT2} contradicts minimality of $V_T(d_{0:T})$, which hence finishes this proof.
\end{proof}

We make the following assumption.
\begin{ass}[Decaying sensitivity]\label{ass:decay}
	There exists $\beta\in\mathcal{KL}$ such that for any $x^\mathrm{i}_1,x^\mathrm{i}_2\in\mathcal{X}$ and $x^\mathrm{t}_1,x^\mathrm{t}_2\in\mathcal{X}$ and any data $d_{\tau:\tau+N}$ for which~\eqref{eq:NLP} is feasible, it holds that
		\begin{align}
			&|\bar{\zeta}_N(j,d_{\tau:\tau+N},x^\mathrm{i}_1,x^\mathrm{t}_1)-\bar{\zeta}_N(j,d_{\tau:\tau+N},x^\mathrm{i}_2,x^\mathrm{t}_2)|\nonumber \\
			&\leq \beta(|x^\mathrm{i}_1-x^\mathrm{i}_2|,{j}) + \beta(|x^\mathrm{t}_1-x^\mathrm{t}_2|,{N-j})\label{eq:decay}
		\end{align}
		for all $j\in\mathbb{I}_{[0,N]}$, $N\in\mathbb{I}_{[0,T]}$, and $\tau\in\mathbb{I}_{[0,T-N]}$.
\end{ass}

For linear systems and quadratic cost functions as in~\eqref{eq:stage_cost} and~\eqref{eq:term_cost}, one can establish Assumption~\ref{ass:decay} (with the $\mathcal{KL}$-function $\beta$ specializing to an exponential one) under observability and controllability (the latter one with respect to the disturbance input) by suitably adapting \cite[Th.~5]{Ferrante2005}, compare also \cite{Gruene2020}.
This can conceptually be transferred to nonlinear systems under (strong) regularity conditions on the solution, cf.~\cite[Th.~7, Sec.~2.2]{Shin2021}.
In our case, the property in~\eqref{eq:decay} involves two different data sets that only differ in $x^\mathrm{i}_1,x^\mathrm{i}_2$ and $x^\mathrm{t}_1,x^\mathrm{t}_2$, but otherwise contain exactly the same data; consequently, the bound involves only these two terms.

\begin{thm}[Decaying sensitivity implies turnpike]\label{thm:turnpike}
	Let $\mathcal{X}$ be compact and suppose that Assumption~\ref{ass:decay} is satisfied.
	Consider the optimal FIE solution ${z}^*_{j} =  \zeta_T(j,d_{0:T})$, $j\in\mathbb{I}_{[0,T]}$ with some given $d_{0:T}$ and $T\in\mathbb{I}_{\geq0}$.
	Then, there exists $C>0$ such that the solution $\hat{z}^*_{\tau+j} = \zeta_N(j,d_{\tau:\tau+N})$ satisfies
	\begin{align}\label{eq:turnpike}
		|\hat{z}_{\tau+j}^*-z_{\tau+j}^*|\leq \beta(C,j) + \beta(C,N-j)
	\end{align}
	for all $j\in\mathbb{I}_{[0,N]}$, $\tau\in\mathbb{I}_{[0,T-N]}$, $N\in\mathbb{I}_{[0,T]}$, and $T\in\mathbb{I}_{\geq0}$, and all possible data $d_{0:T}$.		
\end{thm}
\begin{proof}
	Theorem~\ref{thm:turnpike} is a direct consequence of the fact that $\hat{z}_{\tau:\tau+N}^*$ is a minimizer of $\bar{P}(d_{\tau:\tau+N},\hat{x}^*_\tau,\hat{x}^*_{\tau+N})$, and $z_{\tau:\tau+N}^*$ of $\bar{P}(d_{\tau:\tau+N},x^*_{\tau},x^*_{\tau+N})$ (by Lemma~\ref{lem:optimality}). The application of \eqref{eq:decay} in Assumption~\ref{ass:decay} then leads to $|\hat{z}_{\tau+j}^*-z_{\tau+j}^*|\leq \beta(|\hat{x}^*_{\tau}-x^*_{\tau}|,j) + \beta(|\hat{x}^*_{\tau+N}-x^*_{\tau+N}|,N-j)$ for all $j\in\mathbb{I}_{[0,N]}$ and $\tau\in\mathbb{I}_{[0,T-N]}$. By compactness of $\mathcal{X}$, there exists $C>0$ such that $|x_1-x_2|\leq C$ for all $x_1,x_2\in\mathcal{X}$. In combination, we obtain~\eqref{eq:turnpike}, which finishes this proof.
\end{proof}

The two terms on the right-hand side in~\eqref{eq:turnpike} capture the occurrence of approaching and leaving arcs. Intuitively, these reflect the fact that the truncated problem $P_N$ considers only a segment of the full data batch, neglecting both the past data $d_{0:\tau-1}$ and the future data $d_{\tau+N+1:T}$.
	
\begin{rem}[Turnpike at the boundaries]\label{rem:boundary}
	For the cases where $\tau=0$ and $\tau=T-N$, the problems $P_N$ and $P_T$ (in particular, their data) actually half-side coincide.
	Considering modified problems $\bar{P}$ with either the initial or terminal constraint being enforced, we can theoretically refine the behavior at the left and right boundaries; namely, for $\tau=0$ we can derive that $|\hat{z}_{j}^*-z_{j}^*|\leq \beta(C,N-j),\ j\in\mathbb{I}_{[0,N]}$, and for $\tau=T-N$ that \mbox{$|\hat{z}_{T-N+j}^*-z_{T-N+j}^*|\leq \beta(C,j),\ j\in\mathbb{I}_{[0,N]}$}.
\end{rem}

\begin{rem}[Motivating example revisited]\label{rem:motivation_TP}
	As the system in the motivating example in Section~\ref{sec:mot_example} is linear and trivially observable and controllable, it satisfies Assumption~\ref{ass:decay}; hence, Theorem~\ref{thm:turnpike} applies.
	This fact is particularly evident in the middle plot in Figure~\ref{fig:OE}, which shows clear approaching and leaving arcs.
	Moreover, the left and right plots in Figure~\ref{fig:OE} precisely match the behavior described in Remark~\ref{rem:boundary}, exhibiting only either a leaving arc or an approaching arc for the special cases where $\tau=0$ and~$\tau=T-N$.
\end{rem}

\begin{rem}[Turnpike via strict dissipativity]
	Turnpike arguments are in fact standard in the context of (receding horizon) optimal control, usually closely connected to the concept of strict dissipativity, see, e.g., \cite{Faulwasser2018,Gruene2019,Koehler2019}.
	In fact, we can adapt the standard result that ``strict dissipativity implies turnpike'' (e.g., \cite[Th.~1]{Gruene2018}) to the more general setting of optimal state estimation.
	Here, we can even avoid the need for an additional reachability property to bound the value function, which is due to the conceptual fact that the initial state is free in the estimation problem and not fixed as in control.
	The resulting turnpike property, however, can merely be characterized by a bound on the number of elements of the sequence $\hat{z}^*_{\tau:\tau+N}$ that lie outside an $\epsilon$-neighborhood of the turnpike $z^*_{0:T}$, cf.~\cite{Gruene2016a,Gruene2018}.
	This \emph{measure turnpike} property is conceptually weaker compared to the \emph{$\mathcal{KL}$-turnpike} property in~\eqref{eq:turnpike}, in the sense that it is not possible to infer \emph{which} elements of the solution $\hat{z}^*_{\tau:\tau+N}$ are actually close to the turnpike, and which are not.
	Because this additional information is crucially required in the context of state estimation (as becomes obvious in the following section), we focus on the $\mathcal{KL}$-turnpike characterization as provided by Theorem~\ref{thm:turnpike}.
\end{rem}

\section{Performance of optimal state estimation}\label{sec:performance_sec}
In this section, we show how the FIE problem $P_T$ can be suitably approximated by a sequence of truncated problems $P_N$, and that this approximation is in fact approximately optimal with respect to the full problem (Section~\ref{sec:performance}). Furthermore, we discuss implications for the practically relevant case of MHE for online state estimation (Section~\ref{sec:performance_MHE}).
The key condition for the results in this section is a turnpike property of the following form.

\begin{ass}[Turnpike behavior]\label{ass:turnpike}
	There exists $\beta\in\mathcal{KL}$ and $C>0$ such that for all $T\in\mathbb{I}_{\geq0}$ and $N\in\mathbb{I}_{[0,T]}$, the solutions of the problems $P_T(d_{0:T})$ and $P_N(d_{\tau:\tau+N})$ satisfy
	\begin{equation*}
		|\hat{z}_{\tau+j}^*-z_{\tau+j}^*| \leq
		\begin{cases}
			\beta(C,N-j), & \tau = 0\\
			\beta(C,j) + \beta(C,N-j), & \tau \in \mathbb{I}_{[1,T-N-1]}\\
			\beta(C,j), & \tau = T-N
		\end{cases}
	\end{equation*}
	for all $j\in\mathbb{I}_{[0,N]}$, $\tau\in\mathbb{I}_{[0,T-N]}$, and all possible data $d_{0:T}$.
\end{ass}

Note that Assumption~\ref{ass:turnpike} is satisfied if the optimal state estimation problem exhibits a decaying sensitivity property, see Theorem~\ref{thm:turnpike} and Remark~\ref{rem:boundary}.

\subsection{Approximation properties and performance guarantees}\label{sec:performance}
We aim to approximate the FIE solution $z_{0:T}^*$ using a sequence of truncated problems $P_N$ of length $N\in\mathbb{I}_{\geq0}$ (where for simplicity we assume that $N$ is an even number).
Since the individual truncated problems are completely decoupled from each other, computing this approximation can be parallelized and hence has the potential to significantly save time and resources.
Moreover, we show how a careful concatenation of the solutions to the truncated problems leads to a sequence that is approximately optimal with respect to the (unknown) optimal FIE solution, with error terms that can be made arbitrarily small.
This is very interesting from a theoretical point of view, but also practically relevant in, e.g., large data assimilation problems that appear in geophysics and environmental sciences, see, e.g.,~\cite{Asch2016,Carrassi2018}.

Now, consider a batch of measured input-output data $d_{0:T}$ and the solution $z_{0:T}^*$ of the associated FIE problem $P_T(d_{0:T})$. To approximate the corresponding optimal state trajectory $x_{0:T}^*$, we define the \emph{approximate estimator}
\begin{equation}
	\hat{x}^\mathrm{ae}_j =
	\begin{cases}
		\zeta_N^{x}(j,d_{0:N}), & j \in\mathbb{I}_{[0,N/2]}\\
		\zeta_N^{x}(N/2,d_{j-N/2:j+N/2}), & j\in\mathbb{I}_{[N/2+1,T-N/2-1]}\\
		\zeta_N^{x}(N-T+j,d_{T-N:T}), & j \in\mathbb{I}_{[T-N/2,T]},
	\end{cases}\label{eq:candidate_x}
\end{equation}
where $\zeta_N^x$ is defined below~\eqref{eq:z_N}.
This specific construction ensures that $\hat{x}^\mathrm{ae}_j$ lies in a neighborhood of $x^*_j$ for all $j\in\mathbb{I}_{[0,T]}$, which is shown in the following result.

\begin{prop}\label{prop:candi}
	Suppose that Assumption~\ref{ass:turnpike} holds. Then, the state sequence~$\hat{x}^\mathrm{ae}_{0:T}$ satisfies $|\hat{x}^\mathrm{ae}_j - x^*_j| \leq  2\beta(C,N/2)$ for all $j\in\mathbb{I}_{[0,T]}$ and all possible data $d_{0:T}$.
\end{prop}

\begin{proof}
	The sequence $\hat{x}^\mathrm{ae}_{0:T}$ as defined by \eqref{eq:candidate_x} is essentially constructed from three parts, where for each part we can exploit the turnpike property from Assumption~\ref{ass:turnpike}.
	In particular, for $j\in\mathbb{I}_{[0,N/2]}$, Assumption~\ref{ass:turnpike} with $\tau=0$ yields
	\begin{equation*}
		|\hat{x}^\mathrm{ae}_j - x^*_j| \leq \beta(C,N-j) \leq \beta(C,N/2), \ j\in\mathbb{I}_{[0,N/2]}.
	\end{equation*}	
	For $j\in\mathbb{I}_{[N/2+1,T-N/2-1]}$, we can use Assumption~\ref{ass:turnpike} with $\tau\in\mathbb{I}_{[1,T-N-1]}$:
	\begin{align*}
		|\hat{x}^\mathrm{ae}_j - x^*_j|&\leq  2\beta(C,N/2), \ j\in\mathbb{I}_{[N/2+1,T-N/2-1]}.
	\end{align*}
	Finally, for $j\in\mathbb{I}_{[T-N/2,T]}$, we employ Assumption~\ref{ass:turnpike} with $\tau=T-N$:
	\begin{equation*}
		|\hat{x}^\mathrm{ae}_j - x^*_j| \leq \beta(C,N-T+j) \leq \beta(C,N/2), \ j\in\mathbb{I}_{[T-N/2,T]}.
	\end{equation*}
	Combining these three cases yields the desired result.
\end{proof}

In the following, we consider the special case where the dynamics~\eqref{eq:sys_1} are subject to additive disturbances:
\begin{equation}
	f(x,u,w) = f_\mathrm{a}(x,u) + w.\label{eq:dynamics}
\end{equation}
Here, we further impose a Lipschitz condition on $f_\mathrm{a}$.

\begin{ass}\label{ass:Lipschitz}
	The function $f_\mathrm{a}$ is Lipschitz in $x\in\mathcal{X}$ uniformly in $u\in\mathcal{U}$, i.e., there exists a constant $L_f>0$ such that $|f_\mathrm{a}(x_1,u)-f_\mathrm{a}(x_2,u)| \leq L_f|x_1-x_2|$ for all $x_1,x_2\in\mathcal{X}$ uniformly for all $u\in\mathcal{U}$.
\end{ass}

Note that Assumption~\ref{ass:Lipschitz} is not restrictive under compactness of $\mathcal{X}$ (which we generally require in our results) and compactness of $\mathcal{U}$.
Under the dynamics~\eqref{eq:dynamics}, the sequence~$\hat{x}^\mathrm{ae}_{0:T}$ represents a feasible state trajectory, where the corresponding disturbance inputs are given by
\begin{equation}
	\hat{w}_j^\mathrm{ae} = \hat{x}^\mathrm{ae}_{j+1}-f_\mathrm{a}(\hat{x}^\mathrm{ae}_{j},u_j),\  j\in\mathbb{I}_{[0,T-1]}, \label{eq:candidate_w}
\end{equation}
provided that $\mathcal{W}$ is such that $\hat{w}_j^\mathrm{ae}\in\mathcal{W}$ for all $j\in\mathbb{I}_{[0,T-1]}$.
Define $\hat{z}^\mathrm{ae}_j := (\hat{x}^\mathrm{ae}_j,\hat{w}^\mathrm{ae}_j)$, $j\in\mathbb{I}_{[0,T-1]}$, and $\hat{z}^\mathrm{ae}_T := (\hat{x}^\mathrm{ae}_T,0)$.

We further require a Lipschitz property of the functions $l$ and~$g$ used in the cost function~\eqref{eq:MHE_cost}.

\begin{ass}\label{ass:Lipschitz_gh}
	The functions $l$ and $g$ are Lipschitz on $\mathcal{X}\times\mathcal{W}$ uniformly on $\mathcal{U}\times\mathcal{Y}$, i.e., there exist constants \mbox{$L_l,L_g>0$} such that $|l(x_1,w_1;d){\,-\,}l(x_2,w_2;d)| \leq L_l|(x_1,w_1)-(x_2,w_2)|$ and $|g(x_1;d){\,-\,}g(x_2;d)| \leq L_g|x_1{\,-\,}x_2|$ for all $x_1,x_2\in\mathcal{X}$ and $w_1,w_2\in\mathcal{W}$ uniformly for all $d\in\mathcal{U}\times\mathcal{Y}$.
\end{ass}

Assumption~\ref{ass:Lipschitz_gh} can be easily satisfied in practical applications; in particular, for the common special case of quadratic cost functions as in \eqref{eq:stage_cost} and \eqref{eq:term_cost}, Assumption~\ref{ass:Lipschitz_gh} holds under compactness of $\mathcal{X}$, $\mathcal{W}$, and the domain of the data $\mathcal{U}\times\mathcal{Y}$.

We are now able to state the main result of this section.
\begin{thm}[Performance]\label{thm:perf} 
	Consider the system dynamics~\eqref{eq:dynamics} under Assumption~\ref{ass:Lipschitz} and suppose that Assumptions~\ref{ass:turnpike} and \ref{ass:Lipschitz_gh} hold.
	Then, there exist functions ${\sigma}_1,\sigma_2\in\mathcal{L}$ such that the approximate estimator $\hat{z}^\mathrm{ae}_{0:T}$ as defined via~\eqref{eq:candidate_x} and \eqref{eq:candidate_w} satisfies
	\begin{align}
		J_T(\hat{x}^\mathrm{ae}_{0:T},\hat{w}^\mathrm{ae}_{0:T-1};d_{0:T})\leq V_T(d_{0:T}) + T\sigma_1(N) + \sigma_2(N) \label{eq:thm_performance}
	\end{align}
	for any possible data $d_{0:T}$.
\end{thm}

\begin{proof}
	We start by noting that
	\begin{align}
		&J_T(\hat{x}^\mathrm{ae}_{0:T},\hat{w}^\mathrm{ae}_{0:T-1};d_{0:T})-V_T(d_{0:T})\nonumber\\
		&\leq |J_T(\hat{x}^\mathrm{ae}_{0:T},\hat{w}^\mathrm{ae}_{0:T-1};d_{0:T})-V_T(d_{0:T})|\nonumber\\
		&\leq \sum_{j=0}^{T-1}L_l|(\hat{x}^\mathrm{ae}_{j},\hat{w}^\mathrm{ae}_j)-({x}^*_{j},{w}^*_j)| + L_g|\hat{x}^\mathrm{ae}_{T}-{x}^*_{T}|,\label{eq:proof_start}
	\end{align}
	where in the last inequality we used the definition of the cost function~\eqref{eq:MHE_cost} together with the triangle inequality and Assumption~\ref{ass:Lipschitz_gh}.
	For each $j\in\mathbb{I}_{[0,T-1]}$, again by the triangle inequality it follows that
	\begin{equation}
		|(\hat{x}^\mathrm{ae}_{j},\hat{w}^\mathrm{ae}_j)-({x}^*_{j},{w}^*_j)|
		\leq |\hat{x}^\mathrm{ae}_{j}-{x}^*_{j}| + |\hat{w}^\mathrm{ae}_j - {w}^*_j|. \label{eq:proof_A1}
	\end{equation}
	Using the dynamics~\eqref{eq:dynamics} and Assumption~\ref{ass:Lipschitz} leads to
	\begin{equation}
		|\hat{w}^\mathrm{ae}_j - {w}^*_j| \leq |\hat{x}^\mathrm{ae}_{j+1}  - {x}^*_{j+1}| + L_f|\hat{x}^\mathrm{ae}_{j}- {x}^*_{j}|.\label{eq:proof_A2}
	\end{equation}
	From Proposition~\ref{prop:candi}, we further obtain
	\begin{equation}
		|\hat{x}^\mathrm{ae}_j - x^*_j| \leq |\hat{z}^\mathrm{ae}_j - z^*_j| \leq 2\beta(C,N/2) =: \sigma(N)\label{eq:proof_A3}
	\end{equation}
	for all $j\in\mathbb{I}_{[0,T]}$, where $\sigma\in\mathcal{L}$.
	Combining~\eqref{eq:proof_start}--\eqref{eq:proof_A3} yields
	\begin{align*}
		&J_T(\hat{x}^\mathrm{ae}_{0:T},\hat{w}^\mathrm{ae}_{0:T-1};d_{0:T})-V_T(d_{0:T})\\
		&\leq T(2+L_f)L_l\sigma(N) + L_g\sigma(N).
	\end{align*}
	The definitions $\sigma_1(s):=(2+L_f)L_l\sigma(s)$ and $\sigma_2(s):=L_g\sigma(s)$ for $s\geq0$ establish~\eqref{eq:thm_performance}; noting that the functions $\sigma_1$ and $\sigma_2$ are of class $\mathcal{L}$ concludes this proof.
\end{proof}

Some remarks are in order.

\begin{rem}[Performance estimate] \label{rem:imterpretation}
		 The performance estimate in~\eqref{eq:thm_performance} implies that the approximate estimator $\hat{z}^\mathrm{ae}_{0:T}$ is approximately optimal on $\mathbb{I}_{[0,T]}$ with error terms that depend on the choices of $N$ and $T$.
		 Moreover, in case of an \emph{exponential turnpike} property (i.e., Assumption~\ref{ass:turnpike} holds with $\beta(s,t)=ce^{-kt}$ for some $c,k>0$), the $\mathcal{L}$-functions $\sigma_{1}$ and $\sigma_{2}$ in~\eqref{eq:thm_performance} also decay exponentially. Then, the performance $J_T$ converges \emph{exponentially} to the optimal performance $V_T$ as $N$ increases. This behavior is also evident in the numerical example in Section~\ref{sec:example}.
		 Finally, note that the first error term in~\eqref{eq:thm_performance} grows linearly with $T$ and tends to infinity if $T$ approaches infinity. This property is to be expected (due to the fact that the turnpike is never exactly reached) and conceptually similar to (non-averaged) performance results in economic model predictive control, cf. \cite[Sec.~5]{Faulwasser2018}, \cite{Gruene2016b}.
\end{rem}

To make meaningful statements in the asymptotic case where $T\rightarrow\infty$, we analyze the averaged performance of $\hat{z}^{\mathrm{ae}}_{0:T}$ in the following corollary of Theorem~\ref{thm:perf}.

\begin{cor}[Averaged performance]\label{cor:performance}
	Assume that the conditions of Theorem~\ref{thm:perf} are satisfied.
	Then, the approximate estimator $\hat{z}^\mathrm{ae}_{0:T}$ satisfies the averaged performance estimate
	\begin{align*}
		&\limsup\limits_{T\rightarrow\infty}\frac{1}{T}J_T(\hat{x}^\mathrm{ae}_{0:T},\hat{w}^\mathrm{ae}_{0:T-1};d_{0:T})\\[-0.2ex]
		&\leq \limsup\limits_{T\rightarrow\infty}\frac{1}{T}V_T(d_{0:T}) + \sigma_1(N)
	\end{align*}
	for all possible data $d_{0:T}$.
\end{cor}

From Corollary~\ref{cor:performance}, it follows that the averaged performance of $\hat{z}^\mathrm{ae}_{0:T}$ is approximately optimal, where the error term can be made arbitrarily small by a suitable choice of $N$.

\subsection{Implications for moving horizon estimation}\label{sec:performance_MHE}
In online state estimation, the variable $T$ represents the forward (online) time that continuously grows (in the following, we use lowercase $t$ to indicate this case).
Then, at each time instant $t\in\mathbb{I}_{\geq0}$, one is interested in obtaining an accurate estimate of the current true unknown state $x_t$.
Obviously, solving $P_t(d_{0:t})$ for the desired FIE solution $x^*_t$ is generally infeasible in practice since the problem size also continuously grows with time.
Instead, MHE considers the truncated optimal estimation problem $P_N(d_{t-N:t})$ using the most recent data $d_{t-N:t}$ only, where the horizon length $N\in\mathbb{I}_{\geq0}$ is fixed.
More precisely, the MHE estimate at the current time $t\in\mathbb{I}_{\geq0}$ is given by
\begin{equation}
	\hat{x}_{t}^{\mathrm{mhe}} =
	\begin{cases}
		\zeta_t^{x}(t,d_{0:t}), & t \in\mathbb{I}_{[0,N-1]}\\
		\zeta_N^{x}(N,d_{t-N:t}), & t\in\mathbb{I}_{\geq N}.
	\end{cases}\label{eq:MHE_seq_x}
\end{equation}
From our results in Section~\ref{sec:performance}, we know that $\hat{x}_t^\mathrm{mhe}$ is actually the endpoint of some trajectory that is approximately optimal on $\mathbb{I}_{[0,t]}$ (namely, the sequence $\hat{x}^\mathrm{ae}_{0:t}$ from~\eqref{eq:candidate_x}), and hence indeed constitutes a meaningful estimate that is expected to be close to the desired (unknown) FIE solution ${x}_{t}^* = \zeta_t(t,d_{0:t})$ if $N$ is sufficiently large.
In particular, under the turnpike property from Assumption~\ref{ass:turnpike}, we can explicitly bound their difference by
\begin{equation*}
	|\hat{x}_t^{\mathrm{mhe}}-x^*_t| = |\zeta_N(N,d_{t-N:t})-\zeta_t(t,d_{0:t})| \leq \beta(C,N)
\end{equation*}
for all $t\in\mathbb{I}_{\geq N}$, which in fact can be made arbitrarily small by a suitable choice of~$N$ (note that for $t\in\mathbb{I}_{[0,N-1]}$, the FIE and MHE solutions coincide anyway).

However, we also want to emphasize that the MHE sequence $\hat{x}^{\mathrm{mhe}}_{0:t}$ defined by~\eqref{eq:MHE_seq_x} is \emph{not} approximately optimal on the whole interval $\mathbb{I}_{[0,t]}$ (and therefore turns out to be rather unsuitable for offline estimation), because it is the concatenation of solutions of truncated problems that lie on the leaving arc for all $j\in\mathbb{I}_{[0,t-1]}$, which might actually be far away from the turnpike (the optimal FIE solution). This could also be observed in the following numerical example.

\section{Numerical example}\label{sec:example}

We consider the system
\begin{align*}
	x_1^+ &= x_1 + t_{\Delta}(-2k_1x_1^2+2k_2x_2) + u_1 + w_1,\\
	x_2^+ &= x_2 + t_{\Delta}(k_1x_1^2-k_2x_2) + u_2 + w_2,\\
	y &= x_1 + x_2 + v,
\end{align*}
with $k_1=0.16$, $k_2=0.0064$, and $t_{\Delta}=0.1$.
This corresponds to the batch-reactor example from~\cite[Sec.~5]{Tenny2002} under Euler discretization and with additional controls $u$, disturbances $w$, and measurement noise $v$.
We consider a given data set $d_{0:T}$ of length $T=400$, where the process started at $x_0{\,=\,}[3,0]$, was subject to uniformly distributed disturbances and noise satisfying $w \in \{w\in\mathbb{R}^2: |w_i|\leq 0.05, i=1,2\}$ and $v \in \{v\in\mathbb{R}: |v|\leq 0.5\}$, and where the input $u_j$ was used to periodically empty and refill the reactor such that $x_{j+1}{\,=\,}[3,0]^\top{+\,}w_j$ for all $j{\,=\,}50i$ with $i\in\mathbb{I}_{[1,7]}$ and $u_j=0$ for all $j\neq 50i$.
To reconstruct the unknown state trajectory $x_{0:T}$, we consider the cost function~\eqref{eq:MHE_cost}--\eqref{eq:term_cost} and select $Q=I_2$ and $R=G=1$.
In the following, we compare the optimal FIE solution $z_{0:T}^*$, the proposed approximate estimator $\hat{z}^\mathrm{ae}_{0:T}$~(see \eqref{eq:candidate_x} and~\eqref{eq:candidate_w}), and MHE $\hat{z}_{0:T}^\mathrm{mhe}$ (given by \eqref{eq:MHE_seq_x} together with the disturbance estimates $\hat{w}_j^\mathrm{mhe} = \hat{x}^\mathrm{mhe}_{j+1}-f_\mathrm{a}(\hat{x}^\mathrm{mhe}_{j},u_j)$, $j\in\mathbb{I}_{[0,T-1]}$).

\begin{figure}[t]
	\vspace{1.2ex}
	\includegraphics[width=\columnwidth]{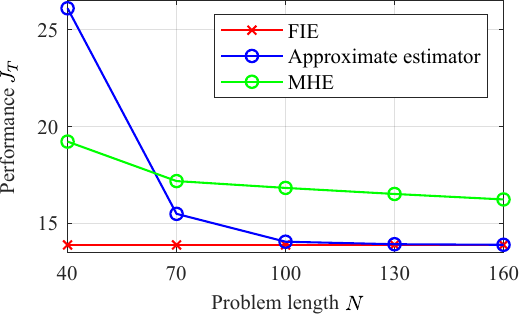}
	\caption{Performance $J_T$ of the approximate estimator $\hat{z}^\mathrm{ae}_{0:T}$ (blue) and MHE $\hat{z}^\mathrm{mhe}_{0:T}$ (green) for different lengths $N$ of the truncated problems $P_N$ compared to the optimal performance of the FIE solution $z^*_{0:T}$ (red).}
	\label{fig:J}
\end{figure}

From Figure~\ref{fig:J}, for small horizons ($N{\,=\,}40$) we observe that the approximate estimator $\hat{z}^\mathrm{ae}_{0:T}$ achieves significantly worse performance compared to FIE (and MHE).
This can be attributed to the problem length~$N$ being too small, leading to the fact that the estimates contained in $\hat{z}^\mathrm{ae}_{0:T}$ correspond to solutions of truncated problems that are far away from the turnpike (the FIE solution), compare also the motivating example in Section~\ref{sec:mot_example}, particularly Figure~\ref{fig:OE} for $N{\,=\,}5$. 
For increasing values of $N$, the estimates are getting closer to the turnpike, and the performance improves significantly.
Specifically, we see \emph{exponential} convergence to the optimal (FIE) performance $V_T$.
This could be expected since the system is exponentially detectable \cite[Sec.~V.A]{Schiller2023c} and controllable with respect to the input $w$, which suggests that Assumption~\ref{ass:turnpike} specializes to exponentially decaying sensitivity, cf.~Remark~\ref{rem:imterpretation}.
Overall, a problem length of $N=130$ is sufficient to achieve nearly optimal performance.
In contrast, the MHE sequence $\hat{z}_{0:T}^\mathrm{mhe}$ generally yields worse performance than the approximate estimator $\hat{z}^\mathrm{ae}_{0:T}$ (for $N\geq70$).
This is in line with our theory, as $\hat{z}_{0:T}^\mathrm{mhe}$ is a concatenation of solutions of truncated problems that are on the right leaving arc and hence may be far from the turnpike, cf.~Section~\ref{sec:performance_MHE}.

In practice, one is usually mainly interested in the accuracy of the state estimates with respect to the \emph{real} unknown system trajectory $x_{0:T}$. To assess this, we compare the sum of the normed errors\footnote{
	For a given sequence $\hat{x}_{0:T}$, we define $\mathrm{SNE}(\hat{x}_{0:T}) := \sum_{j=0}^{T}|\hat{x}_j-x_j|$.
} (SNE) of the FIE solution $x^*_{0:T}$, the approximate estimator $\hat{x}^\mathrm{ae}_{0:T}$, and MHE $x^\mathrm{mhe}_{0:T}$ for different choices of the horizon length $N$.
The corresponding results in Table~\ref{tab:J} show qualitatively the same behavior as in the previous performance analysis. In particular, the FIE solution yields the most accurate estimates with the lowest SNE. The proposed approximate estimator yields much higher SNE for small horizons (SNE increase of 54\,\% for $N=40$ compared to FIE), but improves very fast as $N$ increases, and exponentially converges to the SNE of FIE. On the other hand, the SNE of MHE improves much slower, and is particularly much worse than that of the FIE solution and the proposed approximate estimator (for $N\geq70$).
\begin{table}[t]
	\begin{threeparttable}
		\centering
		\caption{SNE for the proposed approximate estimator and MHE.}
		\label{tab:J}
		\setlength\tabcolsep{5pt}
		\begin{tabular*}{\columnwidth}{@{\extracolsep{\fill}}ccc}
			\toprule
			\ \ Problem length $N$  & Approximate estimator & MHE \\\midrule
			\ \  40 & 92.603 \ (+52.7\,\%) & 71.334 \ (+17.6\,\%) \\
			\ \  70 & 63.417 \ (+4.6\,\%) & 66.588 \ (+9.8\,\%) \\
			\ \ 100 & 61.387 \ (+1.2\,\%) & 65.429 \ (+7.9\,\%) \\
			\ \ 130 & 61.025 \ (+0.6\,\%) & 64.771 \ (+6.8\,\%) \\
			\ \ 160 & 60.805 \ (+0.3\,\%) & 65.007 \ (+7.2\,\%) \\
			\bottomrule
		\end{tabular*}
		\begin{tablenotes}
			\footnotesize
			\item Values in parentheses indicate the relative increase in the SNE compared
			\item to the optimal FIE solution $x^*_{0:T}$ (which achieves SNE\,=\,60.638).\\[-1ex]
		\end{tablenotes}
	\end{threeparttable}
\end{table}

\section{Conclusion}
We showed that the optimal FIE solution serves as turnpike for the solutions of truncated state estimation problems involving only a subset of the measurement data.
We proposed a method to approximate the optimal FIE solution by a sequence of truncated problems (which can be solved in parallel), and we showed that the resulting performance is approximately optimal.
The numerical example illustrates that the performance of the proposed approximate estimator exponentially converges to the optimal performance if the problem length is increased (which is not the case for MHE).
In our follow-up work \cite{Schiller2025}, we extend our theoretical results to the case of MHE for online state estimation.

\input{references.bbl}

\end{document}

%% file: references.bbl

%% file: root.bbl
\begin{thebibliography}{10}
\providecommand{\url}[1]{#1}
\csname url@samestyle\endcsname
\providecommand{\newblock}{\relax}
\providecommand{\bibinfo}[2]{#2}
\providecommand{\BIBentrySTDinterwordspacing}{\spaceskip=0pt\relax}
\providecommand{\BIBentryALTinterwordstretchfactor}{4}
\providecommand{\BIBentryALTinterwordspacing}{\spaceskip=\fontdimen2\font plus
\BIBentryALTinterwordstretchfactor\fontdimen3\font minus
  \fontdimen4\font\relax}
\providecommand{\BIBforeignlanguage}[2]{{%
\expandafter\ifx\csname l@#1\endcsname\relax
\typeout{** WARNING: IEEEtran.bst: No hyphenation pattern has been}%
\typeout{** loaded for the language `#1'. Using the pattern for}%
\typeout{** the default language instead.}%
\else
\language=\csname l@#1\endcsname
\fi
#2}}
\providecommand{\BIBdecl}{\relax}
\BIBdecl

\bibitem{Rawlings2017}
J.~B. Rawlings, D.~Q. Mayne, and M.~Diehl, \emph{{Model Predictive Control:
  Theory, Computation, and Design}}, 2nd~ed.\hskip 1em plus 0.5em minus
  0.4em\relax Santa Barbara, CA, USA: Nob Hill Publish., LLC, 2020, 3rd
  printing.

\bibitem{Allan2021a}
D.~A. Allan and J.~B. Rawlings, ``Robust stability of full information
  estimation,'' \emph{{SIAM} J. Control Optim.}, vol.~59, no.~5, pp.
  3472--3497, 2021.

\bibitem{Knuefer2023}
S.~Knüfer and M.~A. Müller, ``Nonlinear full information and moving horizon
  estimation: {R}obust global asymptotic stability,'' \emph{Automatica}, vol.
  150, p. 110603, 2023.

\bibitem{Schiller2023c}
J.~D. Schiller, S.~Muntwiler, J.~Köhler, M.~N. Zeilinger, and M.~A. Müller,
  ``A {L}yapunov function for robust stability of moving horizon estimation,''
  \emph{{IEEE} Trans. Autom. Control}, vol.~68, no.~12, pp. 7466--7481, 2023.

\bibitem{Hu2023}
W.~Hu, ``Generic stability implication from full information estimation to
  moving-horizon estimation,'' \emph{{IEEE} Trans. Autom. Control}, vol.~69,
  no.~2, pp. 1164--1170, 2024.

\bibitem{Sabag2022}
O.~Sabag and B.~Hassibi, ``Regret-optimal filtering for prediction and
  estimation,'' \emph{IEEE Trans. Signal Process.}, vol.~70, pp. 5012--5024,
  2022.

\bibitem{Brouillon2023}
J.-S. Brouillon, F.~Dörfler, and G.~Ferrari-Trecate, ``Minimal regret state
  estimation of time-varying systems,'' \emph{IFAC-PapersOnLine}, vol.~56,
  no.~2, pp. 2595--2600, 2023.

\bibitem{McKenzie1986}
L.~W. McKenzie, ``Optimal economic growth, turnpike theorems and comparative
  dynamics,'' in \emph{Handbook of Mathematical Economics}, K.~J. Arrow and
  M.~D. Intriligator, Eds.\hskip 1em plus 0.5em minus 0.4em\relax Amsterdam,
  The Netherlands: Elsevier, 1986, vol.~3, pp. 1281--1355.

\bibitem{Carlson1991}
D.~A. Carlson, A.~B. Haurie, and A.~Leizarowitz, \emph{{Infinite Horizon
  Optimal Control}}.\hskip 1em plus 0.5em minus 0.4em\relax Berlin, Heidelberg,
  Germany: Springer, 1991.

\bibitem{Gruene2016b}
L.~Grüne, ``Approximation properties of receding horizon optimal control,''
  \emph{Jahresber. Dtsch. Math. Ver.}, vol. 118, no.~1, pp. 3--37, 2016.

\bibitem{Faulwasser2022}
T.~Faulwasser and L.~Grüne, ``Turnpike properties in optimal control,'' in
  \emph{Handbook of Numerical Analysis}, E.~Trélat and E.~Zuazua, Eds.\hskip
  1em plus 0.5em minus 0.4em\relax Amsterdam, The Netherlands: Elsevier, 2022,
  vol.~23, pp. 367--400.

\bibitem{Gruene2019}
L.~Grüne and S.~Pirkelmann, ``Economic model predictive control for
  time‐varying system: {P}erformance and stability results,'' \emph{Optim.
  Control Appl. Methods}, vol.~41, no.~1, pp. 42--64, 2019.

\bibitem{Gruene2016a}
L.~Grüne and M.~A. Müller, ``On the relation between strict dissipativity and
  turnpike properties,'' \emph{Syst. Control Lett.}, vol.~90, pp. 45--53, 2016.

\bibitem{Trelat2023}
E.~Trélat, ``Linear turnpike theorem,'' \emph{Math. Control Signals Syst.},
  vol.~35, no.~3, pp. 685--739, 2023.

\bibitem{Allan2020a}
D.~A. Allan, ``A {L}yapunov-like function for analysis of model predictive
  control and moving horizon estimation,'' Ph.D. dissertation, Univ.
  Wisconsin-Madison, 2020.

\bibitem{Koehler2019}
J.~Köhler, M.~A. Müller, and F.~Allgöwer, ``Nonlinear reference tracking:
  {A}n economic model predictive control perspective,'' \emph{IEEE Trans.
  Autom. Control}, vol.~64, no.~1, pp. 254--269, 2019.

\bibitem{Na2020}
S.~Na and M.~Anitescu, ``Exponential decay in the sensitivity analysis of
  nonlinear dynamic programming,'' \emph{SIAM J. Optim.}, vol.~30, no.~2, pp.
  1527--1554, 2020.

\bibitem{Shin2022}
S.~Shin, M.~Anitescu, and V.~M. Zavala, ``Exponential decay of sensitivity in
  graph-structured nonlinear programs,'' \emph{SIAM J. Optim.}, vol.~32, no.~2,
  pp. 1156--1183, 2022.

\bibitem{Ferrante2005}
A.~Ferrante and L.~Ntogramatzidis, ``Employing the algebraic {R}iccati equation
  for a parametrization of the solutions of the finite-horizon {LQ} problem:
  the discrete-time case,'' \emph{Syst. Control Lett.}, vol.~54, no.~7, pp.
  693--703, 2005.

\bibitem{Gruene2020}
L.~Grüne, M.~Schaller, and A.~Schiela, ``Exponential sensitivity and turnpike
  analysis for linear quadratic optimal control of general evolution
  equations,'' \emph{J. Differ. Equ.}, vol. 268, no.~12, pp. 7311--7341, 2020.

\bibitem{Shin2021}
S.~Shin and V.~M. Zavala, ``Controllability and observability imply exponential
  decay of sensitivity in dynamic optimization,'' \emph{IFAC-PapersOnLine},
  vol.~54, no.~6, pp. 179--184, 2021.

\bibitem{Faulwasser2018}
T.~Faulwasser, L.~Grüne, and M.~A. Müller, ``Economic nonlinear model
  predictive control,'' \emph{Found. Trends Syst. Control}, vol.~5, no.~1, pp.
  1--98, 2018.

\bibitem{Gruene2018}
L.~Grüne, S.~Pirkelmann, and M.~Stieler, ``Strict dissipativity implies
  turnpike behavior for time-varying discrete time optimal control problems,''
  in \emph{Control Systems and Mathematical Methods in Economics},
  G.~Feichtinger, R.~Kovacevic, and G.~Tragler, Eds.\hskip 1em plus 0.5em minus
  0.4em\relax Cham, Germany: Springer, 2018, pp. 195--218.

\bibitem{Asch2016}
M.~Asch, M.~Bocquet, and M.~Nodet, \emph{{Data Assimilation: Methods,
  Algorithms, and Applications}}.\hskip 1em plus 0.5em minus 0.4em\relax
  Philadelphia, PA, USA: SIAM, 2016.

\bibitem{Carrassi2018}
A.~Carrassi, M.~Bocquet, L.~Bertino, and G.~Evensen, ``Data assimilation in the
  geosciences: {A}n overview of methods, issues, and perspectives,''
  \emph{WIREs Climate Change}, vol.~9, no.~5, 2018.

\bibitem{Tenny2002}
M.~J. Tenny and J.~B. Rawlings, ``Efficient moving horizon estimation and
  nonlinear model predictive control,'' in \emph{Proc. Am. Control Conf.},
  2002, pp. 4475--4480.

\bibitem{Schiller2025}
J.~D. Schiller, L.~Grüne, and M.~A. Müller, ``Performance guarantees for
  optimization-based state estimation using turnpike properties,''
  \emph{arXiv:2501.18385}, 2025.

\end{thebibliography}
